\documentclass[11pt,a4paper]{article}
\usepackage{amsmath,amssymb,amsthm,amsfonts,latexsym,graphicx,subfigure}
\usepackage[all,import]{xy}
\usepackage[numbers,sort&compress]{natbib}
\usepackage{indentfirst}
\usepackage{fancyhdr,enumerate}
\usepackage{bbm,color}
\usepackage{tikz}
\usetikzlibrary{shapes.geometric, arrows}
\usepackage{accents,cases}
\usepackage{multirow}
\usepackage[lined, ruled, linesnumbered, longend]{algorithm2e}
\usepackage{hyperref}

\usepackage{array,float}
\newcommand{\PreserveBackslash}[1]{\let\temp=\\#1\let\\=\temp}
\newcolumntype{C}[1]{>{\PreserveBackslash\centering}p{#1}}
\newcolumntype{R}[1]{>{\PreserveBackslash\raggedleft}p{#1}}
\newcolumntype{L}[1]{>{\PreserveBackslash\raggedright}p{#1}}

\makeatletter
\def\wbar{\accentset{{\cc@style\underline{\mskip8mu}}}}

\makeatother

\theoremstyle{plain}
\newtheorem{theorem}{Theorem}

\newtheorem{lemma}{Lemma}

\allowdisplaybreaks

\begin{document}

\title{Existence results for Tzitz\'eica equation via topological degree method on graphs}

\author{Kaizhe Chen\footnotemark[1]\and Heng Zhang\footnotemark[2]}

\footnotetext[1]{School of Gifted Young, University of Science and Technology of China, Hefei 230026, China. \\
Email address:
{\tt ckz22000259@mail.ustc.edu.cn} }

\footnotetext[2]{School of Mathematical Sciences, 
University of Science and Technology of China, Hefei 230026, China. \\
Email address:
{\tt  hengz@mail.ustc.edu.cn} }

\date{}\maketitle
\begin{abstract}
    We derive some existence results for the solutions of the Tzitz\'eica equation
    \begin{equation*}
    -\Delta u + h_1(x)e^{Au} + h_2(x)e^{-Bu}=0
    \end{equation*}
    and the  generalized Tzitz\'eica equation
    \begin{equation*}
    -\Delta u + h_1(x)e^{Au}(e^{Au}-1)+h_2(x)e^{-Bu}(e^{-Bu}-1)=0
    \end{equation*}
    on any connected finite graph \(G=(V, E)\). Here, \(h_1(x)>0\), \(h_2(x)>0\) are two given functions on \(V\), and \(A, B>0\) are two constants. Our approach involves computing the topological degree and using the connection between the degree and the critical group of an associated functional.

\end{abstract}

\section{Introduction}
Our research is motivated by the recent work of Jevnikar and Yang \cite{JY17}. In their previous work, they considered the following equation:
\begin{equation} \label{eq:1}
	\Delta u + h_1 e^u - h_2 e^{-2u} = 0 \qquad \mathrm{in}~B_1 \subset \mathbb{R}^2,
\end{equation}
where \(h_1\) and \(h_2\) are smooth positive functions, and \(B_1\) denotes the unit ball in \(\mathbb{R}^2\). Jevnikar and Yang provided the quantization of local blow-up masses associated with a sequence of solutions exhibiting blow-up behavior and established the general existence results for equation \eqref{eq:1}.

Equation \eqref{eq:1} is closely related to the Tzitz\'eica equation, which arises in differential geometry in the context of surfaces with constant affine curvature (see, for example, \cite{tz1, tz2, tz3}). Additionally, it is also related to Euler's equation for one-dimensional ideal gas dynamics, the Hirota--Satsuma equation, and others (see, for example, \cite{DP09, G1984, HR1980, HS1976}).

Unlike Euclidean spaces, the analysis on graphs has a wide range of applications, including neural networks, digital image processing, numerical computation, and more. Recently, partial differential equations (PDEs) arising in geometry and physics have been extensively studied. Grigor'yan, Lin, and Yang \cite{G1, G2, G3} first investigated the existence of solutions for a series of nonlinear elliptic equations on graphs using variational methods and the mountain pass theorem. For other related existence results of PDEs on graphs obtained via variational methods, we refer to \cite{Ge18, GJ18, GHS23, HS21, HSZ20, HS22, Hua1, Hua2, HLY20, HWY21, ZZ18}.

Similar to variational methods, topological degree theory is a powerful tool for studying the existence of solutions to partial differential equations in Euclidean space or on Riemann surfaces. For example, see Li \cite{Liyanyan}. Topological degree theory was first employed by Sun and Wang \cite{SW22} to prove the existence and non-existence of solutions for the Kazdan-Warner equation on finite graphs. For additional related results on the existence of PDEs on graphs using topological degree methods, we refer to \cite{Liu22, LSY24, LY24, CLWWY24, HQ24}.

To state the Tzitz\'eica equations in the graph setting, we review some notations.
  Let $G=(V,E)$ be a connected finite graph, where $V$ is the set of vertices and $E$ is the set of edges. Let
   $\mu :V\rightarrow (0,+\infty)$ and $\{w_{xy}:xy\in E\}$ be its measure and weights, respectively.
   We assume that the weights $w_{xy}$ on edges $xy \in E$ are always positive and symmetric. The Laplacian of a function $u: V\rightarrow \mathbb{R}$ reads as
$$\Delta u(x)=\frac{1}{\mu (x)} \sum_{y\sim x} w_{xy}(u(y)-u(x)),$$
where $y\sim x$ means  $xy\in E$. For two functions $u$ and $v$, the associated gradient form reads
$$\Gamma(u, v)(x) = \frac{1}{2\mu(x)} \sum_{y \sim x} w_{xy} (u(y) - u(x))(v(y) - v(x)).$$
Write $\Gamma(u) = \Gamma(u, u)$. We denote the length of its gradient by
\[
|\nabla u|(x) = \sqrt{\Gamma(u)(x)} = \left(\frac{1}{2\mu(x)} \sum_{y \sim x} w_{xy}(u(y) - u(x))^2 \right)^{1/2}.
\]
The integral of any function $u:V\rightarrow \mathbb{R}$ is given by
$$\int_{V}ud\mu=\sum_{x\in V}\mu(x)u(x),$$
and the integral average of $u$ is defined as
$$\overline{u}=\frac{1}{\text{Vol}(G)}\int_{V}ud\mu=\frac{1}{\text{Vol}(G)}\sum_{x\in V}\mu(x)u(x),$$
where $\text{Vol}(G)=\sum_{x\in V}\mu(x)$ is the volume of $G$. Denote by $L^{\infty}(V)$  the space of all functions $f:V\rightarrow\mathbb{R}$ with finite norm $\|f\|_{\infty}$ which is 
defined as $\|f\|_{\infty}=\max_{x\in V} |f(x)|$.

We are interested in the existence results of the Tzitz\'eica equation on finite graphs, namely \begin{equation}\label{eq:T1}
    -\Delta u + h_1(x)e^{Au} + h_2(x)e^{-Bu}=0,
\end{equation}
where  \(h_1(x)\) and \(h_2(x)\) are given functions on \(V\), with \(h_1(x)>0\), and \(A, B>0\) are two constants. 
We consider the map $\mathcal{F}_{h_1, h_2}:L^{\infty}(V)\rightarrow L^{\infty}(V)$ defined by
\begin{equation}\label{deg1}
    \mathcal{F}_{h_1, h_2}(u)= -\Delta u + h_1(x)e^{Au} + h_2(x)e^{-Bu},
\end{equation}
Denote by $B_R=\{u\in L^{\infty}(V): ||u||_{\infty}<R\}$. We conclude by Lemma \ref{priori1} that the Brouwer degree $\text{deg}(\mathcal{F}_{h_1, h_2}, B_R, 0)$ is well-defined for $R>K$ large. Then by the homotopic invariance, $\text{deg}(\mathcal{F}_{h_1, h_2}, B_R, 0)$ is independent of $R$. We define the topological degree by $$\mathbf{d}_{h_1, h_2}=\lim_{R\rightarrow+\infty}\text{deg}(\mathcal{F}_{h_1, h_2}, B_R, 0).$$ For a more comprehensive understanding of the Brouwer degree, including its key properties such as homotopic invariance and Kronecker existence, we refer the readers to Chang \cite[Chapter 3]{Chang05}.

For the topological degree $\mathbf{d}_{h_1, h_2}$, we have the following results.
\begin{theorem}\label{deg:T1}
    Let $h_1(x)$ be a given positive function on $V$, and $A$, $B>0$ be two constants. Then $$
    \mathbf{d}_{h_1, h_2} = \begin{cases}
1, & \max_V h_2(x)<0, \\
0, & \min_V h_2(x)>0. \\
\end{cases}
$$
\end{theorem}

Hence, Theorem \ref{deg:T1} and the Kronecker existence show that the Tzitz\'eica equation has at least one solution if $\max_V h_2(x)<0$. From the  process of proving Lemma \ref{thm1part2}, we have that if $\max_V h_2(x)>0$, then the Tzitz\'eica equation has no solutions.

Further, we consider the following generalized Tzitz\'eica equation:
    \begin{equation}\label{eq:T2}
        -\Delta u + h_1(x)e^{Au}(e^{Au}-1)+h_2(x)e^{-Bu}(e^{-Bu}-1)=0.
     \end{equation}
It also can be seen as a generalization of the Chern-Simons Higgs model, which has been studied in \cite{LSY24}. We consider the map $\mathcal{G}_{h_1, h_2}:L^{\infty}(V)\rightarrow L^{\infty}(V)$ defined by
\begin{equation}\label{deg2}
    \mathcal{G}_{h_1, h_2}(u)= -\Delta u + h_1(x)e^{Au}(e^{Au}-1) + h_2(x)e^{-Bu}(e^{-Bu}-1)
\end{equation}
Similar as before, we conclude by Lemma \ref{priori2} that the Brouwer degree $\text{deg}(\mathcal{G}_{h_1, h_2}, B_R, 0)$ is well defined for $R>L$ large. Then by the homotopic invariance, $\text{deg}(\mathcal{G}_{h_1, h_2}, B_R, 0)$ is independent of $R$. We define the topological degree by $$\mathbf{\tilde{d}}_{h_1, h_2}=\lim_{R\rightarrow+\infty}\text{deg}(\mathcal{G}_{h_1, h_2}, B_R, 0).$$

As for the topological degree $\mathbf{\tilde{d}}_{h_1, h_2}$, we have the following results.

\begin{theorem}\label{deg:T2}
    Let $h_1(x)>0$, $h_2(x)>0$ be two given functions on $V$, and $A$, $B>0$ be two constants. Then $$
    \mathbf{\tilde{d}}_{h_1, h_2} =0.
$$
\end{theorem}
 By taking the variational method and critical groups into account, we  find the multiplicities of solutions of \eqref{eq:T2} under certain conditions.

\begin{theorem}\label{multisolutions}
    Let $h_1(x)>0$, $h_2(x)>0$ be two given functions on $V$, and $A$, $B>0$ be two constants. If $A\max_V h_1(x)<B\min_Vh_2(x)$ or $A\min_V h_1(x)>B\max_V h_2(x)$,  then the equation \eqref{eq:T2} has at least two solutions.
\end{theorem}

\section{Proof of Theorem \ref{deg:T1}}
Firstly, we give the following elliptic estimate.
\begin{lemma}\label{ellestimate}
    Let $G=(V,E)$ be a finite connected simple graph, then there exists a positive constant $C$ such that for all $u \in L^{\infty}(V)$, $$\max_Vu-\min_V u\leq C\|\Delta u\|_{L^\infty(V)}.$$
\end{lemma}

\begin{proof}
    Without loss of generality, we assume that  $u(x_1)=\min_Vu$, $u(x_{\ell})=\max_Vu$ and $x_1x_2$, $x_2x_3$, $\cdots$, $x_{\ell-1}x_\ell$ is the shortest path connecting $x_1$ and $x_\ell$. It follows from Cauchy inequality that
 \begin{eqnarray}\nonumber
0\leq u(x_\ell)-u(x_1)&\leq&\sum_{j=1}^{\ell-1}|u(x_{j+1})-u(x_{j})|\\\nonumber
&\leq&\frac{\sqrt{\ell-1}}{\sqrt{w_0}}\left(\sum_{j=1}^{\ell-1}w_{x_{j+1}x_j}(u(x_{j+1})-u(x_j))^2\right)^{1/2}\\\label{ineq-0}
&\leq&\frac{\sqrt{\ell-1}}{\sqrt{w_0}}\left(\int_V|\nabla u|^2d\mu\right)^{1/2},
 \end{eqnarray}
where $w_0=\min_{x\in V,\,y\sim x} w_{xy}>0$. Denoting $\overline{u}=\frac{1}{|V|}\int_Vud\mu$, we obtain by integration by parts
\begin{eqnarray*}
\int_V|\nabla u|^2d\mu&=&-\int_Vu\Delta ud\mu\\
&=&-\int_V(u-\overline{u})\Delta ud\mu\\
&\leq&\left(\int_V(u-\overline{u})^2d\mu\right)^{1/2}\left(\int_V(\Delta u)^2d\mu\right)^{1/2}\\
&\leq&\left(\frac{1}{\lambda_1}\int_V|\nabla u|^2d\mu\right)^{1/2}\left(\int_V(\Delta u)^2d\mu\right)^{1/2},
\end{eqnarray*}
which gives
\begin{align}\label{ineq-1}\int_V|\nabla u|^2d\mu\leq \frac{1}{\lambda_1}\int_V(\Delta u)^2d\mu\leq \frac{1}{\lambda_1}
\|\Delta u\|_{L^\infty(V)}^2\text{Vol}(G),\end{align}
where $\lambda_1=\inf_{\overline{v}=0,\int_Vv^2d\mu=1}\int_V|\nabla v|^2d\mu>0$. Combining (\ref{ineq-0}) and (\ref{ineq-1}), we conclude
\begin{align}\label{equiv}
  \max_Vu-\min_V u\leq \sqrt{\frac{(\ell-1)\text{Vol}(G)}{w_0\lambda_1}}\|\Delta u\|_{L^\infty(V)}:=C\|\Delta u\|_{L^\infty(V)},
\end{align} as desired.
\end{proof}

Secondly, we give the following priori estimate for equation \eqref{eq:T1}.
\begin{lemma}\label{priori1}
    Let \( u \) be a solution of equation \eqref{eq:T1}, where \( A \) and \( B \) are positive constants, \( h_1(x) > 0 \) and $h_2(x) < 0$ are two functions on $V$. Then there exists a constant \( K \), depending only on \( A \), \( B \), \( h_1 \), \( h_2 \), and the graph \( G \), such that \( |u(x)| \leq K \) for all \( x \in V \).
\end{lemma}

\begin{proof}
    Let $x_0 \in V$ be a vertex such that $u(x_0)=\min_V u(x) $. Then 
    \begin{align*}
        -\Delta u(x_0)=\frac{1}{\mu_{x_0}} \sum_{y \in V} w_{x_0 y}(u(x_0)-u(y)) \leq0.
    \end{align*}
     Hence, we have $$h_1(x_0)e^{Au(x_0)}+h_2(x_0)e^{-Bu(x_0)}\geq 0.$$
    That is, $$u(x_0) \geq \frac{1}{A+B} \log \frac{-h_2(x_0)}{h_1(x_0)} \geq \frac{1}{A+B} \log \frac{-\max_V h_2(x)}{\max_V h_1(x)} := C_1.$$
    Similarly, for a vertex $x_1 \in V$  such that $u(x_1)=\max_V u(x) $, we have $$u(x_1)\leq C_2 :=\frac{1}{A+B} \log \frac{-\min_V h_2(x)}{\min_V h_1(x)}.$$ We obtain the desired estimate.
\end{proof}

\begin{lemma}\label{thm1part1}
    In the case of \(h_1(x) > 0 \) and \(h_2(x) < 0 \), we have $\mathbf{d}_{h_1, h_2}=1$.
\end{lemma}

\begin{proof}
We define a smooth map $\tilde{\mathcal{F}}: L^{\infty}(V)\times[0,1]\rightarrow L^{\infty}(V)$ by $$\tilde{\mathcal{F}}(u,t)= -\Delta u + [t\epsilon+(1-t)h_1(x)]e^{Au} + [-t\epsilon+(1-t)h_2(x)]e^{-Bu}.$$
    where $\epsilon>0$ is sufficiently small such that $t\epsilon+(1-t)h_1(x)>0$ and $-t\epsilon+(1-t)h_2(x)<0$. Note that $\tilde{\mathcal{F}}(u,0)=\mathcal{F}_{h_1, h_2}.$ Applying Lemma \ref{priori1}, we deduce that the solutions of $\tilde{\mathcal{F}}(u,t)=0$ are uniformly bounded with respect to $t$. We consider the  equation
    \begin{equation}\label{defor-u1}
        \tilde{\mathcal{F}}(u,1)=-\Delta u + \epsilon e^{Au}-\epsilon e^{-Bu}=0.
    \end{equation}
By homotopic invariance, the equation \eqref{eq:T1} and the equation \eqref{defor-u1} have the same Brouwer degree. Thus, it is sufficient to calculate the Brouwer degree of the equation \eqref{defor-u1}.

We claim that $u_1 \equiv 0$ is the unique solution of equation \eqref{defor-u1}. Suppose $v_1 \not\equiv 0$ is  another solution of equation \eqref{defor-u1}, then $$-\Delta (u_1-v_1) + \epsilon (e^{Au_1}-e^{Av_1})-\epsilon (e^{-Bu_1}-e^{-Bv_1})=0,$$ which implies that
\begin{align*}
\Delta (u_1-v_1)|\leq& \epsilon |e^{Au_1}-e^{Av_1}|+\epsilon |e^{-Bu_1}-e^{-Bv_1}|\\ \leq& \epsilon(A+B)|u_1-v_1|\cdot \max\{e^{Au_1},e^{Av_1},e^{-Bu_1},e^{-Bv_1}\}\\ :=&
C_3\epsilon|u_1-v_1|,
\end{align*}
here we have used Lemma \ref{priori1}.
Together  with  Lemma \ref{ellestimate}, we derive
\begin{align*}
    \max_V(u_1-v_1)-\min_V(u_1-v_1)&\leq CC_3\epsilon \max_V|u_1-v_1|\\ &:=C_4\epsilon \max_V|u_1-v_1|\\ &\leq C_4\epsilon\left(\max_V(u_1-v_1)-\min_V(u_1-v_1)+|\min_V(u_1-v_1)|\right).
\end{align*}
Thus for sufficient small $\epsilon>0$, we have
$$\max_V(u_1-v_1)\leq \min_V(u_1-v_1)+\frac{1}{2}|\min_V(u_1-v_1)|.$$
Since $$\int_{V} (e^{Au_1}-e^{-Bu_1})- (e^{Av_1}-e^{-Bv_1}) d\mu =0,$$ it must hold that $$\min_V (u_1-v_1)<0<\max_V (u_1-v_1).$$
Hence, $$0< \max_V (u_1-v_1) \leq \min_V (u_1-v_1)-\frac{1}{2}\min_V (u_1-v_1)=\frac{1}{2}\min_V (u_1-v_1)<0,$$ which is a contradiction. Therefore, the  equation \eqref{defor-u1} has the unique solution $u \equiv 0$.

A straightforward calculation shows that $$\mathbf{d}_{h_1, h_2}=\mathbf{d}_{\epsilon, -\epsilon}=\text{sgn\;det}(D\mathcal{F}_{\epsilon,-\epsilon}(0))=\text{sgn\;det}(-\Delta+A\epsilon I+B\epsilon I)=1,$$ for sufficient small $\epsilon>0$, where we have used the fact that the eigenvalues of $-\Delta$ are non-negative. This completes the proof.
\end{proof}

\begin{lemma}\label{thm1part2}
    In the case of \(h_1(x) > 0 \) and \(h_2(x) > 0 \), we have $\mathbf{d}_{h_1, h_2}=0$.
\end{lemma}

\begin{proof}
    Integrating the both side of \eqref{eq:T1}, we deduce that the left side is positive whereas the right side is $0$. Hence the equation \eqref{eq:T1} has no solution. Therefore, we conclude $$\mathbf{d}_{h_1, h_2}=0,$$ as desired.
\end{proof}
Combining Lemmas \ref{thm1part1} and \ref{thm1part2}, we complete the proof of Theorem \ref{deg:T1}.

\section{Proof of Theorem \ref{deg:T2}}

In order to ensure that the Brouwer degree is well-defined and to utilize the homotopy invariance, we prove the following estimate.
\begin{lemma}\label{priori3}
    Let \( u_t \) be a solution of equation \begin{equation}\label{defor:T2}
        -\Delta u_t + h_1(x)e^{Au_t}(e^{Au_t}-t)+h_2(x)e^{-Bu_t}(e^{-Bu_t}-t)=0,
     \end{equation} where $t \in [0,1]$. Then there exists a constant \( L \), depending only on \( A \), \( B \), \( h_1 \), \( h_2 \), and the graph \( G \), such that \( |u_t(x)| \leq L \) for all \( x \in V \) and all $t \in [0,1]$.
\end{lemma}
\begin{proof}
    Firstly, we prove that $u_t(x)$ has a uniform upper bound. Fix $t \in [0,1]$. Let $x_0 \in V$ be a vertex such that $u_t(x_0)=\max_V u_t(x) $. Then 
    \begin{align*}
        -\Delta u_t(x_0)=\frac{1}{\mu_{x_0}} \sum_{y \in V} w_{x_0 y}(u_t(x_0)-u_t(y))  \geq0.
    \end{align*} Therefore, we have $$h_1(x_0)e^{Au_t(x_0)}(e^{Au_t(x_0)}-t)+h_2(x_0)e^{-Bu_t(x_0)}(e^{-Bu_t(x_0)}-t)\leq 0.$$
Thus, $$e^{Au_t(x_0)}(e^{Au_t(x_0)}-t)\leq -\frac{h_2(x_0)}{h_1(x_0)}e^{-Bu_t(x_0)}(e^{-Bu_t(x_0)}-t)\leq \frac{t^2\max_V h_2(x)}{4\min_V h_1(x)},$$
which implies that $$ \max_V u(x) \leq \frac{1}{A} \log\left(\frac{t}{2}+\sqrt{\frac{t^2\max_V h_2(x)}{4\min_V h_1(x)}+\frac{t^2}{4}}\right) \leq \frac{1}{A} \log\left(\frac{1}{2}+\sqrt{\frac{\max_V h_2(x)}{4\min_V h_1(x)}+\frac{1}{4}}\right):= C_1.$$

Secondly, we show that $u_t(x)$ has also a uniform lower bound. Without loss of generality, we may assume $\min_V u_t(x)<0$. For otherwise, $u_t(x)$ has already lower bound $0$. We derive from integrating both sides of equation \eqref{defor:T2} that
$$\int_{V}h_1(x)e^{Au_t(x)}(e^{Au_t(x)}-t)+h_2(x)e^{-Bu_t(x)}(e^{-Bu_t(x)}-t)d\mu = 0,$$
which leads to \begin{align*}
    &\left| \int_{u_t<0} h_2(x) e^{-Bu_t(x)} \left( e^{-Bu_t(x)} - t \right) d\mu \right| \\
=& \left| \int_{u_t\geq0} h_2(x) e^{-Bu_t(x)} \left( e^{-Bu_t(x)} - t \right) d\mu + \int_{V} h_1(x) e^{Au_t(x)} \left( e^{Au_t(x)} - t \right) d\mu \right| \\
\leq& \max_V h_2(x) \int_{u_t\geq0} \left|e^{-Bu_t(x)} \left( e^{-Bu_t(x)} - t \right) \right|d\mu  +  \max_V h_1(x)\int_{V}  \left|e^{Au_t(x)} \left( e^{Au_t(x)} - t \right) \right|d\mu .
\end{align*}
Observe that $$ \int_{u_t\geq0} \left|e^{-Bu_t(x)} \left( e^{-Bu_t(x)} - t \right) \right|d\mu  \leq \max\left\{\frac{t^2}{4} , 1-t\right\} \text{Vol}(G) \leq \text{Vol}(G),$$ and $$ \int_{V}  \left|e^{Au_t(x)} \left( e^{Au_t(x)} - t \right) \right|d\mu  \leq \max \left\{\frac{t^2}{4}, e^{AC_1} \left( e^{AC_1} - t \right)\right\}\text{Vol}(G).$$
This together with the fact that$$\left| \int_{u_t<0} h_2(x)e^{-Bu_t(x)} \left( e^{-Bu_t(x)} - t \right) d\mu \right| \geq \min_V h_2(x) \cdot e^{-B\min_V u_t(x)} \left( e^{-B\min_V u_t(x)} - t \right) \mu_0$$ leads to \begin{align*}
    &\min_V h_2(x) \cdot e^{-B\min_V u_t(x)} \left( e^{-B\min_V u_t(x)} - t \right) \mu_0\\ \leq & \max_V h_2(x) \text{Vol}(G) + \max_V h_1(x) \cdot \max \left\{\frac{t^2}{4}, e^{AC_1} \left( e^{AC_1} - t \right)\right\}\text{Vol}(G)\\ \leq & \max_V h_2(x) \text{Vol}(G) + \max_V h_1(x) \cdot \max \left\{\frac{1}{4}, e^{2AC_1} \right\}\text{Vol}(G):= C_2,
\end{align*}
 where $\mu_0 =\min_{x \in V}\mu (x)>0$. Hence, there holds
$$ \min_V u_t(x) \geq -\frac{1}{B} \log \left(\frac{t}{2}+\sqrt{\frac{C_2}{\min_V h_2(x)\mu_0}+\frac{t^2}{4}}\right)\geq -\frac{1}{B} \log \left(\frac{1}{2}+\sqrt{\frac{C_2}{\min_V h_2(x)\mu_0}+\frac{1}{4}}\right):=C_3.$$This completes the proof.
\end{proof}

In particular, we have the follow result.
\begin{lemma}
\label{priori2}
    Let \( u \) be a solution of equation \eqref{eq:T2}, where \( A \) and \( B \) are positive constants, \( h_1(x) > 0 \), and \( h_2(x)> 0 \). Then there exists a constant \( L \), depending only on \( A \), \( B \), \( h_1 \), \( h_2 \), and the graph \( G \), such that \( |u(x)| \leq L \) for all \( x \in V \).
\end{lemma}

Now, we are prepared to calculate the Brouwer degree.

\begin{lemma}\label{pri:defor:T2}
   Let \( A \) and \( B \) be two positive constants, \( h_1(x) > 0 \), and  \(  h_2(x) > 0 \). Then $$\mathbf{\tilde{d}}_{h_1, h_2}=0.$$
\end{lemma}

\begin{proof}
    We define a smooth map $\tilde{\mathcal{G}}: L^{\infty}(V)\times[0,1]\rightarrow L^{\infty}(V)$ by $$\tilde{\mathcal{G}}(u,t)= -\Delta u + h_1(x)e^{Au}(e^{Au}-t)+h_2(x)e^{-Bu}(e^{-Bu}-t).$$ By Lemma \ref{priori3} and homotopic invariance, we derive $$\mathbf{\tilde{d}}_{h_1, h_2}=\text{deg}(\tilde{\mathcal{G}}(u,1), B_R, 0)=\text{deg}(\tilde{\mathcal{G}}(u,0), B_R, 0)=\mathbf{d}_{h_1, h_2}=0,$$as desired.
\end{proof}

\section{Proof of Theorem \ref{multisolutions}}
In this section, we adopt the variational method to find a solution to equation \eqref{eq:T2} with the assumption in Theorem \ref{multisolutions}.

\begin{lemma}
    Let $h_1(x)>0$, $h_2(x)>0$ be two given functions on $V$, and $A$, $B>0$ be two constants. If $A\max_V h_1(x)<B\min_V h_2(x)$ or $A\min_V h_1(x)>B\max_V h_2(x)$,  then the equation \eqref{eq:T2} has at least one solution.
\end{lemma}

\begin{proof}
        If $A\max_V h_1(x)<B\min_V h_2(x)$, then there exists a small $\delta>0$ such that $$\frac{h_{1}(x)}{h_2(x)}e^{(A+B)\delta}<-\frac{e^{-B\delta}-1}{e^{A\delta}-1}$$ for all $x\in V$. That is, 
        \begin{align}\label{eq:Gdelta}
            \mathcal{G}_{h_1, h_2}(\delta)=h_1(x)e^{A\delta}(e^{A\delta}-1)+h_2(x)e^{-B\delta}(e^{-B\delta}-1)<0, \ \ \text{for\  all} \ \ x\in V .
        \end{align}
       Obviously, there exists a sufficiently large constant $\beta>0$ such that 
       \begin{align}\label{eq:Gbeta}
           \mathcal{G}_{h_1, h_2}(\beta)=h_1(x)e^{A\beta}(e^{A\beta}-1)+h_2(x)e^{-B\beta}(e^{-B\beta}-1)>0, \ \ \text{for\  all} \ \ x\in V .
       \end{align}
       
        We define the energy functional $J: L^{\infty}(V)\rightarrow\mathbb{R}$ as 
        \begin{equation}\label{energyfunc}
            J(u)=\frac{1}{2}\int_V|\nabla u|^2+\frac{1}{A}h_1(x)(e^{Au}-1)^2-\frac{1}{B}h_2(x)(e^{-Bu}-1)^2d\mu.
        \end{equation}
        It is direct to check $J \in C^{2}(L^{\infty}(V), \mathbb{R})$, since $L^{\infty}(V)$ is a finite-dimensional Banach space. As the set $\left\{u \in L^{\infty}(V): \delta \leq u(x) \leq \beta \ \text{for\  all} \ \ x\in V\right\} $ is bounded and closed in $L^{\infty}(V)$, we find some $\tilde{u} \in L^{\infty}(V)$ satisfying $\delta \leq \tilde{u} \leq \beta$ for all $x \in V$ and 
        \begin{align}\label{eq:Jmin}
            J(\tilde{u})=\min_{\delta \leq u \leq \beta}J(u).
        \end{align}
        We {\it claim} that
\begin{equation}\label{eq:region}
    \delta< \tilde{u}(x)< \beta \text{\ for \,all\ \,}  x\in V.
\end{equation}
In fact, if $\tilde{u}(x_0)=\delta$, then we take a small $\epsilon>0$ such that
$$\delta \leq \tilde{u}(x)+t\delta_{x_0}(x) \leq \beta,\quad\forall x\in V,\,\forall t\in(0,\epsilon),$$
where $\delta_{x_0}(x)$ denotes the dirac function at $x_0$. Combining with \eqref{eq:Gdelta} and \eqref{eq:Jmin}, we have
\begin{align}\label{eq:pertuba}
  \notag  0 \leq& \left. \frac{d}{dt} J(\tilde{u} + t \delta_{x_0}) \right|_{t=0}\\
   \notag   =&\int_V\left(-\Delta \tilde{u}+h_1(x) e^{A\tilde{u}}(e^{A\tilde{u}}-1)+h_2(x) e^{-B\tilde{u}}(e^{-B\tilde{u}}-1)\right)\delta_{x_0}d\mu\\
   \notag   =&-\Delta \tilde{u}(x_0)+h_1(x) e^{A\tilde{u}(x_0)}(e^{A\tilde{u}(x_0)}-1)+h_2(x) e^{-B\tilde{u}(x_0)}(e^{-B\tilde{u}(x_0)}-1)\\
    \notag  =&-\Delta \tilde{u}(x_0)+\mathcal{G}_{h_1, h_2}(\delta)\\
    <&-\Delta \tilde{u}(x_0).
\end{align}

On the other hand, since $\tilde{u}(x)\geq \tilde{u}(x_0)$ for all $x\in V$, we conclude $\Delta \tilde{u}(x_0)\geq 0$,
which contradicts \eqref{eq:pertuba}. Hence, $\tilde{u}(x)>\delta$ for all $x\in V$. Similarly, we can also rule out the possibility that $\tilde{u}(x_1)=\beta$ for some $x_1\in V$. This confirms our claim. 

Combining \eqref{eq:Jmin} and \eqref{eq:region}, we conclude
that $\tilde{u}$ is a local minimum critical point of $J$. In particular, $\tilde{u}$ is a solution of the Tzitz\'eica equation \eqref{eq:T2}. The similar argument holds true if $A\min_V h_1(x)>B\max_V h_2(x)$. This completes the proof.
\end{proof}

\emph{Proof of Theorem \ref{multisolutions}.}

Let $\tilde{u}$ be a local minimum critical point of $J$ that is defined as \eqref{energyfunc}. With no loss of generality, we may assume $\tilde{u}$ is the unique
critical point of $J$. For otherwise, $J$ has already at least two critical points, and the proof ends. Set $a:=J(\tilde{u})$.
Recall the $q$-th critical group of ${J}$ at $\tilde{u}$ is defined by (\cite{Chang93}, Chapter 1, Definition 4.1)
 \begin{align}\label{group}
    \mathsf{C}_q({J},\tilde{u})=\mathsf{H}_q({J}^a\cap {U},{J}^a\setminus\{\tilde{u}\}
    \cap {U},\mathbf{G}),
 \end{align}
 where ${J}^a=\{u\in L^\infty(V):{J}(u)\leq a\}$, $U$ is a neighborhood of
 $\tilde{u}\in L^\infty(V)$, $\mathsf{H}_q$
 is the singular homology group with the coefficients groups $\mathbf{G}$. By the excision property of
 $\mathsf{H}_q$, we derive that the definition of $\mathsf{C}_q({J},\tilde{u})$ is independent on the choice of $U$. 
 Since $\tilde{u}$ is the unique local minimum critical point, we get ${J}^a=\{\tilde{u}\}$. Therefore, we deduce
\begin{align}\label{cri}
\notag C_q({J},\tilde{u}) &= \mathsf{H}_q(J^a \cap U, \{J^a \setminus \{\tilde{u}\} \cap U; \mathbf{G}) \\
\notag &= \mathsf{H}_q(\{\tilde{u}\}, \varnothing; \mathbf{G}) \\
\notag &= \mathsf{H}_q(\{\tilde{u}\}; \mathbf{G}) \\
&= \delta_{q0} \mathbf{G},
\end{align}
where
\begin{align*}
\delta_{q0} \mathbf{G} = 
\begin{cases} 
\mathbf{G}, & q = 0, \\
0, & q > 0.
\end{cases}
\end{align*}
We claim that $J$ satisfies the Palais-Smale condition at any value $c \in \mathbb{R}$. Indeed, let $\{u_k\}$ be a sequence such that $J(u_k)\rightarrow c\in\mathbb{R}$ and $J^\prime(u_k)(\phi)\rightarrow 0$ for any $\phi \in L^{\infty}(V)$
as $k\rightarrow\infty$. Using the method of proving Lemma \ref{priori3}, we obtain that $\{u_k\}$ is uniformly bounded. Since $L^{\infty}(V)$ is pre-compact,
then up to a subsequence, $\{u_k\}$ converges uniformly to some $u^\ast \in L^{\infty}(V)$ which is a critical point of $J$, as claimed. 
Notice  that
$$DJ(u)=-\Delta u + h_1(x)e^{Au}(e^{Au}-1)+h_2(x)e^{-Bu}(e^{-Bu}-1)=\mathcal{G}_{h_1,h_2}(u).$$
According to (\cite{Chang93}, Chapter 2, Theorem 3.2), based on (\ref{cri}), we have, for sufficiently large $R$,
$$\deg(\mathcal{G}_{h_1,h_2},B_{R},0)=\deg(DJ,B_R,0)=\sum_{q=0}^\infty (-1)^q{\rm rank}\, \mathsf{C}_q(J,\tilde{u})={\rm rank}\mathbf{G}\neq 0.$$
This contradicts $\deg(\mathcal{G}_{h_1,h_2},B_{R},0)=0$ derived from Theorem \ref{deg:T2}. Therefore the equation \eqref{eq:T2} has at least two different solutions,
and the proof of Theorem \ref{multisolutions} is finished. $\hfill\Box$

\section*{Acknowledgement}
This work is supported by the New Lotus Scholars Program PB22000259. H.Z. is very grateful to Yijian Zhang  for his helpful discussions.


\begin{thebibliography}{99}
   \bibitem{Carffarali} L. Caffarelli and Y. S. Yang, Vortex condensation in the Chern-Simons Higgs model: An existence theorem. \emph{Comm. Math. Phys.} 168 (1995), 321--336.

\bibitem{Chang93} K. C. Chang, \emph{Infinite dimensional Morse theory and multiple solution problems}. Birkh\"auser, Boston, 1993.

\bibitem{Chang05} K. C. Chang, \emph{Methods in nonlinear analysis}. Springer Monographs in Mathematics. Springer-Verlag, Berlin, 2005.

\bibitem{CH23} R. Chao and S. Hou, Multiple solutions for a generalized Chern-Simons equation on graphs. \emph{J. Math. Anal. Appl.} 519 (2023), Paper No. 126787.

\bibitem{CLWWY24} L. Cui, Y. Liu, C. Wang, J. Wang, and W. Yang, The Einstein-scalar field Lichnerowicz equations on graphs. \emph{Calc. Var. Partial Differential Equations} 63 (2024), no. 6, Paper No. 138, 45 pp.

\bibitem{DP09} M. Dunajski and P. Plansangkate, Strominger-Yau-Zaslow Geometry, Affine Spheres and Painlev\'e III. \emph{Comm. Math. Phys.} 290 (2009), 997--1024.

\bibitem{G1984} B. Gaffet, $SU(3)$ symmetry of the equations of unidimensional gas flow, with arbitrary entropy distribution. \emph{J. Math. Phys.} 25 (1984), 245--255.

\bibitem{HS21} X. L. Han and M. Q. Shao, $p$-Laplacian equations on locally finite graphs. \emph{Acta Math. Sin. (Engl. Ser.)} 37(11) (2021), 1645--1678.

\bibitem{HSZ20} X. L. Han, M. Q. Shao, and L. Zhao, Existence and convergence of solutions for nonlinear biharmonic equations on graphs. \emph{J. Differential Equations} 268(7) (2020), 3936--3961.

\bibitem{Ge18} H. Ge, A $p$-th Yamabe equation on graph. \emph{Proc. Amer. Math. Soc.} 146 (2018), 2219--2224.

\bibitem{GJ18} H. Ge and W. Jiang, Kazdan-Warner equation on infinite graphs. \emph{J. Korean Math. Soc.} 55 (2018), 1091--1101.

\bibitem{G1} A. Grigor'yan, Y. Lin, and Y. Yang, Kazdan-Warner equation on graph. \emph{Calc. Var. Partial Differential Equations} 55 (2016), no. 4, Paper No. 92, 13 pp.

\bibitem{G2} A. Grigor'yan, Y. Lin, and Y. Yang, Yamabe type equations on graphs. \emph{J. Differential Equations} 261(9) (2016), 4924--4943.

\bibitem{G3} A. Grigor'yan, Y. Lin, and Y. Yang, Existence of positive solutions to some nonlinear equations on locally finite graphs. \emph{Sci. China Math.} 60(7) (2017), 1311--1324.

\bibitem{GHS23} Q. Gu, X. Huang, and Y. Sun, Semi-linear elliptic inequalities on weighted graphs. \emph{Calc. Var. Partial Differential Equations} 62 (2023), no. 2, Paper No. 42, 14 pp.

\bibitem{HR1980} R. Hirota and A. Ramani, The Miura transformations of Kaup's equation and of Mikhailov's equation. \emph{Phys. Lett. A} 76 (1980), 95--96.

\bibitem{HS1976} R. Hirota and J. Satsuma, $N$-soliton solutions of model equations for shallow water waves. \emph{J. Phys. Soc. Japan} 40 (1976), 611--612.

\bibitem{HQ24} S. Hou and W. Qiao, Solutions to a generalized Chern-Simons Higgs model on finite graphs by topological degree. \emph{J. Math. Phys.} 65 (2024), no. 8, Paper No. 081503, 11 pp.

\bibitem{HS22} S. Hou and J. Sun, Existence of solutions to Chern-Simons-Higgs equations on graphs. \emph{Calc. Var. Partial Differential Equations} 61 (2022), no. 4, Paper No. 139, 13 pp.

\bibitem{Hua1} B. Hua, R. Li, and L. Wang, A class of semilinear elliptic equations on groups of polynomial growth. \emph{J. Differential Equations} 363 (2023), 327--349.

\bibitem{Hua2} B. Hua and W. Xu, Existence of ground state solutions to some nonlinear Schrödinger equations on lattice graphs. \emph{Calc. Var. Partial Differential Equations} 62 (2023), no. 4, Paper No. 127, 17 pp.

\bibitem{HLY20} A. Huang, Y. Lin, and S. T. Yau, Existence of solutions to mean field equations on graphs. \emph{Comm. Math. Phys.} 377 (2020), 613--621.

\bibitem{HWY21} H. Huang, J. Wang, and W. Yang, Mean field equation and relativistic Abelian Chern-Simons model on finite graphs. \emph{J. Funct. Anal.} 281 (2021), Paper No. 109218, 36 pp.

\bibitem{JWY18} A. Jevnikar, J. Wei, and W. Yang, On the topological degree of the mean field equation with two parameters. \emph{Indiana Univ. Math. J.} 67 (2018), no. 1, 29--88.

\bibitem{JY17} A. Jevnikar and W. Yang, Analytic aspects of the Tzitz\'eica equation: blow-up analysis and existence results. \emph{Calc. Var. Partial Differential Equations} 56 (2017), no. 2, Paper No. 43, 23 pp.

\bibitem{LSY24} J. Li, L. Sun, and Y. Yang, Topological degree for Chern-Simons Higgs models on finite graphs. \emph{Calc. Var. Partial Differential Equations} 63 (2024), no. 4, Paper No. 81, 21 pp.

\bibitem{Liyanyan} Y. Y. Li, Harnack type inequality: the method of moving planes. \emph{Comm. Math. Phys.} 200 (1999), 421--444.

\bibitem{Liu22} Y. Liu, Brouwer degree for mean field equation on graph. \emph{Bull. Korean Math. Soc.} 59(5) (2022), 1305--1315.

\bibitem{LY24} Y. Liu and Y. Yang, Topological degree for Kazdan-Warner equation in the negative case on finite graph. \emph{Ann. Global Anal. Geom.} 65 (2024), no. 4, Paper No. 29, 20 pp.

\bibitem{SW22} L. Sun and L. Wang, Brouwer degree for Kazdan-Warner equations on a connected finite graph. \emph{Adv. Math.} 404 (2022), Part B, Paper No. 108422, 29 pp.

\bibitem{tz1} G. Tzitz\'eica, Sur une nouvelle classe de surfaces. \emph{Rend. Circ. Mat. Palermo} 25 (1908), 180--187.

\bibitem{tz2} G. Tzitz\'eica, Sur une nouvelle classe de surfaces. \emph{C. R. Acad. Sci. Paris} 150 (1910), 955--956.

\bibitem{tz3} G. Tzitz\'eica, \emph{G\'eom\'etrie diff\'erentielle projective des r\'eseaux}. Gauthier-Villars, Paris, 1924.

\bibitem{ZZ18}
N. Zhang and L. Zhao. Convergence of ground state solutions for nonlinear Schr\"{o}dinger equations on graphs. \emph{Sci. China Math.} 61(8) (2018), 1481--1494.
\end{thebibliography}
\end{document}